\documentclass[11pt, a4paper]{amsart}

\usepackage[euler-digits]{eulervm}

\usepackage{amsfonts, amsthm, amssymb, amsmath, stmaryrd}
\usepackage{mathrsfs,array}
\usepackage{eucal,fullpage,times,color,enumerate,accents}
\usepackage{url}

\usepackage[backref]{hyperref}
\hypersetup{
  colorlinks   = true,          
  urlcolor     = blue,          
  linkcolor    = blue,          
  citecolor   = purple             
}

\usepackage{color}
\usepackage{mathrsfs}
\usepackage{amssymb}
\usepackage{tikz}
\usepackage{tikz-cd}
\usepackage{bm}
\usepackage{enumerate}
\usetikzlibrary{calc}

\title{\large Realization of groups with pairing as Jacobians of finite graphs}
\author{Louis Gaudet, David Jensen, Dhruv Ranganathan, Nicholas Wawrykow, and Theodore Weisman}
\date{\today}

\address{{\bf Louis Gaudet}\newline Department of Mathematics, Rutgers University}
\email{lmg289@rutgers.edu}

\address{{\bf David Jensen} \newline Department of Mathematics, University of Kentucky}
\email{djensen@uky.edu}

\address{{\bf Dhruv Ranganathan}\newline Department of Mathematics, Massachusetts Institute of Technology}
\email{dhruvr@mit.edu}

\address{{\bf Nicholas Wawrykow}\newline Department of Mathematics, University of Michigan}
\email{wawrykow@umich.edu}

\address{{\bf Theodore Weisman}\newline Department of Mathematics, University of Texas}
\email{weisman@math.utexas.edu}

\newtheorem*{theorem*}{Theorem}
\newtheorem{theorem}{Theorem}
\newtheorem{corollary}[theorem]{Corollary}
\newtheorem{lemma}[theorem]{Lemma}
\newtheorem{proposition}[theorem]{Proposition}
\newtheorem{definition}[theorem]{Definition}

\newtheorem{conjecture}[theorem]{Conjecture}

\newtheorem{construction}{Construction}

\newtheorem{quasi-theorem}[theorem]{Quasi-Theorem}

\newtheorem{rem1}[theorem]{Remark}
\newenvironment{remark}{\begin{rem1}\em}{\end{rem1}}

\newtheorem{not1}[theorem]{Notation}

\newcommand{\QQ} {{\mathbb Q}}		
		
\newcommand{\ZZ} {{\mathbb Z}}	

\newcommand{\Jac}{\operatorname{Jac}}
\newcommand{\ddiv}{\operatorname{div}}
\newcommand{\ord}{\operatorname{ord}}

\DeclareMathOperator{\val}{val}

\newcommand{\AP}{\mathcal{A}_{2^r}}
\newcommand{\BP}{\mathcal{B}_{2^r}}
\newcommand{\CP}{\mathcal{C}_{2^r}}
\newcommand{\DP}{\mathcal{D}_{2^r}}
\newcommand{\EP}{\mathcal{E}_{2^r}}
\newcommand{\FP}{\mathcal{F}_{2^r}}

\tikzset{me/.style={to path={
\pgfextra{%
 \pgfmathsetmacro{\startf}{-(#1-1)/2}
 \pgfmathsetmacro{\endf}{-\startf}
 \pgfmathsetmacro{\stepf}{\startf+1}}
 \ifnum 1=#1 -- (\tikztotarget)  \else
     let \p{mid}=($(\tikztostart)!0.5!(\tikztotarget)$)
         in
\foreach \i in {\startf,\stepf,...,\endf}
    {%
     (\tikztostart) .. controls ($ (\p{mid})!\i*6pt!90:(\tikztotarget) $) .. (\tikztotarget)
      }
      \fi
     \tikztonodes
}}}

\begin{document}

\pagestyle{plain}
\maketitle

\begin{abstract}
We study which groups with pairing can occur as the Jacobian of a finite graph. We provide explicit constructions of graphs whose Jacobian realizes a large fraction of odd groups with a given pairing. Conditional on the generalized Riemann hypothesis, these constructions yield all groups with pairing of odd order, and unconditionally, they yield all groups with pairing whose prime factors are sufficiently large. For groups with pairing of even order, we provide a partial answer to this question, for a certain restricted class of pairings. Finally, we explore which finite abelian groups occur as the Jacobian of a simple graph. There exist infinite families of finite abelian groups that do not occur as the Jacobians of simple graphs.
\end{abstract}

\section{Introduction}

Given a finite graph $G$, there is naturally associated group $Jac(G)$, the \textit{Jacobian} of $G$. The group $\Gamma = Jac(G)$ comes with a symmetric, bilinear, non-degenerate pairing~\cite{Lorenzini2000, Shokrieh2010},
\[
\langle \cdot,\cdot \rangle: \Gamma\times \Gamma \to \QQ/\ZZ,
\]
known as the \textit{monodromy pairing}. Groups with such a pairing will be referred to simply as \textit{groups with pairing}. Clancy, Leake, and Payne~\cite{CLP} observed that the Jacobian of a randomly generated graph is cyclic with probability close to $0.79$. This probability agrees with the well-known Cohen--Lenstra heuristics, which predict that a finite abelian group $\Gamma$ should occur with probability proportional to $\frac{1}{\vert \mathrm{Aut}(\Gamma) \vert}$. However, other classes of groups violate these heuristics. This is because the Jacobian of a graph should really be thought of as a group, together with a duality pairing. In loc.cit., it is conjectured that a group with pairing $(\Gamma , \langle \cdot , \cdot \rangle )$ should occur with probability proportional to $\frac{1}{\vert \Gamma \vert \vert \mathrm{Aut}(\Gamma , \langle \cdot , \cdot \rangle ) \vert}$.  This is further suggested by the empirical evidence of \cite{CKLPW} and proven in \cite{W14}.

Given a finite abelian group with pairing $\Gamma$, the probability that a random graph has Jacobian isomorphic to $\Gamma$ is zero \cite{W14}, so it is possible that some groups with pairing do not occur at all.  In the present text, we investigate precisely which finite abelian groups with pairing can occur as the Jacobian of a finite graph.  Our main result is the following.

\begin{theorem}\label{thm: nonconditional-pairing}
Let $\Gamma$ be a finite abelian group with pairing. There exists a finite set of primes $\mathscr P\subset \ZZ$ such that, if $\vert \Gamma \vert$ is not divisible by any $p\in \mathscr P$, then there exists a graph $G$ such that
\[
\Gamma \cong Jac(G)
\]
as groups with pairing.
\end{theorem}

It is our expectation that the set of primes $\mathscr P$ appearing in Theorem~\ref{thm: nonconditional-pairing} consists of only the prime 2.  We have the following result, conditional on the generalized Riemann hypothesis~\cite{Dav00}.

\begin{theorem}[Conditional on GRH]\label{thm: conditional-pairing}
Let $\Gamma$ be a finite abelian group with pairing of odd order.  Then there exists a graph $G$ such that
\[
\Gamma \cong Jac(G)
\]
as groups with pairing.
\end{theorem}

\begin{remark}
The above results are related to the following purely number theoretic question. \textit{Given a prime $p$, does there exist a prime $q<2\sqrt{p}$, with $q\equiv 3 \mod 4$, such that $q$ is a quadratic non-residue modulo $p$?}  Numerical evidence suggests that this condition should be satisfied for all sufficiently large primes $p$.
\end{remark}

An interesting variation on the question considered here was studied by Bosch and Lorenzini in~\cite[Proposition 5.2]{BoschLorenzini}. They consider the representation of groups with pairing arising from \textit{arithmetical graphs}. While the strategy of our proof bears some similarities to that found in loc. cit., the presence of arithmetical structure simplifies the classification problem.  Indeed, as shown in~\cite[Example 5.4]{BoschLorenzini}, in the case of arithmetical graphs one can take the underlying graph to be a tree. Our setting is motivated by considerations in tropical geometry and the graph theoretic Abel--Jacobi theory of Baker and Norine.

Jacobians of wedge-sums of graphs decompose canonically as the orthogonal direct sum of the Jacobians of their components. A structure theorem for groups with pairing therefore allows us to focus primarily on the case where $\Gamma$ is cyclic.  When $\Gamma$ is a $2$-group, however, this structure result is more complicated. There are $4$ \textit{non-exceptional} natural pairings on the group $\ZZ/2^r\ZZ$, and we find graphs which realize these groups with pairings. There are, in addition, $2$ exceptional families of pairings on the group $(\ZZ/2^r\ZZ)^2$ that do not decompose as the orthogonal direct sum of cyclic groups with pairing. We refer to Section~\ref{sec: background} for background regarding pairings on $2$-groups.

\begin{theorem}
Let $\Gamma \cong (\ZZ/2^r\ZZ,\langle\cdot,\cdot\rangle)$ be a cyclic $2$-group with non-exceptional pairing $\langle\cdot,\cdot\rangle$.  Then there exists a graph $G$ such that
\[
\Gamma \cong Jac(G)
\]
as groups with pairing.
\end{theorem}

We discuss groups with exceptional pairings in further detail in Section~\ref{sec: exceptional-pairings}.

If we forget the structure of the pairing on $\Gamma$, it is elementary to observe that every finite abelian group $\Gamma$ occurs as the Jacobian of a multigraph $G$.  Naively, however, the construction often necessitates the use of graphs with multiple edges.  Since the Erd\H{o}s--R\'enyi random graphs studied in~\cite{CKLPW,CLP,W14} are always simple, we find it natural to ask the following.

\noindent
{\bf Question.} Which finite abelian groups (without a specified pairing) occur as the Jacobian of a simple graph?

We find that there are infinite families of finite groups that do not occur as the Jacobians of simple graphs.

\begin{theorem}\label{thm: z2k}
For any $k\geq 1$, there exists no simple graph $G$ such that
\[
Jac(G)\cong (\ZZ/2\ZZ)^k.
\]
\end{theorem}

More generally, we have the following result for groups with a large number of $\ZZ/2\ZZ$ invariant factors.

\begin{theorem}\label{thm:too-many-twos}
Let $H$ be a finite abelian group. Then there exists a natural number $k_H$ depending on $H$, such that for all $k>k_H$, there does not exist a simple graph $G$ with
\[
Jac(G)\cong (\ZZ/2\ZZ)^k\times H.
\]
\end{theorem}

\subsection*{Acknowledgements} This project was completed as part of the 2014 Summer Undergraduate Mathematics Research at Yale (SUMRY) program, where the second and third authors were supported as mentors and the first, fourth, and fifth authors were supported as participants. It is a pleasure to thank all involved in the program for creating a vibrant research community. We benefited from conversations with Dan Corey, Andrew Deveau, Jenna Kainic, Nathan Kaplan, Susie Kimport, Dan Mitropolsky, and Anup Rao. We thank Sam Payne for suggesting the problem. We are also especially grateful to Paul Pollack, whose ideas significantly strengthened the results of this paper. Finally, we thank the referees for their careful reading and insightful comments. \\

\noindent
The authors were supported by NSF grant CAREER DMS-1149054 (PI: Sam Payne).

\section{Background}\label{sec: background}

\subsection{Jacobians of graphs}
We briefly recall the basics of divisor theory on graphs. We refer to~\cite{BN07} for further details. In this paper a \textit{graph} will mean a finite connected graph, possibly with multiple edges, but without loops at vertices. A \textit{simple graph} is a graph without multiple edges. A \textit{divisor} on a graph is an integral linear combination of vertices, and we write a divisor as
\[
D = \sum_{v \in V(G)} D(v)v,
\]
where each $D(v)$ is an integer.  The \textit{degree} of a divisor $D$ is
\[
\deg (D) = \sum_{v \in V(G)} D(v).
\]
It is common to think of a divisor as a configuration of ``chips'' and ``anti-chips'' on the vertices of the graph, so that the degree is just the total number of chips.

Let $\mathcal{M}(G) := \mathrm{Hom} (V(G),\ZZ )$ be the group of integer-valued functions on the vertices of $G$.  For $f \in \mathcal{M}(G)$, we define
$$ \ord_v (f) := \sum_{e = vw \newline \text{ edge containing } v} (f(v)-f(w)) ,$$
and
$$ \ddiv (f) := \sum_{v \in V(G)} \ord_v (f) v . $$
Divisors that arise as $\ddiv(f)$ for a function $f\in \mathcal M(G)$ are referred to as \textit{principal}.  We say that two divisors $D_1$ and $D_2$ are \textit{equivalent}, and write $D_1 \sim D_2$, if their difference is principal.

Equivalence of divisors is related to the well-known ``chip-firing game'' on graphs, which can be described as follows.  Given a divisor $D$ and a vertex $v$, the \textit{chip-firing move} centered at $v$ corresponds to the vertex $v$ giving one chip to each of its neighbors. That is, the vertex $v$ loses a number of chips equal to its valence, and each neighbor gains exactly $1$ chip. Two divisors are equivalent if one can be obtained from the other by a sequence of chip-firing moves.

Note that the degree of a divisor is invariant under equivalence.  The \textit{Jacobian} $\Jac (G)$ is the group of equivalence classes of divisors of degree zero.  The Jacobian of a connected graph is always a finite group, with order equal to the number of spanning trees in $G$, see~\cite{BS13}.

For the most part, we will not need any deep structural results about the Jacobians of graphs.  The following result, however, will greatly simplify one of our proofs in the later sections.

\begin{theorem}{\cite[Theorem 2]{CR2000}}
\label{thm:graph_dual_iso}
Let $G$ be a planar graph and let $G^\star$ be a planar dual of $G$. Then, the Jacobian of $G$ and $G^\star$ are isomorphic as groups.
\end{theorem}

The Jacobian of a graph comes equipped with a bilinear pairing, known as the \textit{monodromy pairing}, defined as follows.  Given two divisors $D_1 , D_2 \in \Jac (G)$, first find an integer $m$ such that $mD_1$ is principal -- that is, there exists a function $f \in \mathcal{M}(G)$ such that $\ddiv (f) = mD_1$.  Then we define
$$
\langle D_1 , D_2 \rangle = \frac{1}{m} \sum_{v \in V(G)} D_2(v) f(v) .
$$

It is of course not immediately clear that the pairing above is non-degenerate. A proof may be found in~\cite[Theorem~3.4]{Shokrieh2010}.

\begin{remark} 
Note that the isomorphism of Jacobians of planar dual graphs does \textit{not} in general preserve the pairings. See for instance Corollary~\ref{cor: PairingBananaAndCycle}.
\end{remark}

\subsection{Reduced divisors and Dhar's burning algorithm}
Given a divisor $D$ and a vertex $v_0$, we say that $D$ is \textit{$v_0$-reduced} if
\begin{enumerate}
\item  $D(v) \geq 0$ for all vertices $v \neq v_0$, and
\item  every non-empty set $A \subseteq V(G) \smallsetminus \{ v_0 \}$ contains a vertex $v$ such that $\mathrm{outdeg}_A (v) > D(v)$.
\end{enumerate}
By \cite[Proposition~3.1]{BN07}, every divisor is equivalent to a unique $v_0$-reduced divisor.

There is a simple algorithm for determining whether a given divisor satisfying (1) above is $v_0$-reduced, known as \textit{Dhar's burning algorithm}.  For $v \neq v_0$, imagine that there are $D(v)$ buckets of water at $v$.  Now, light a fire at $v_0$.  The fire consumes the graph, burning an edge if one of its endpoints is burnt, and burning a vertex $v$ if the number of burnt edges adjacent to $v$ is greater than $D(v)$ (that is, there is not enough water to fight the fire).  The divisor $D$ is $v_0$-reduced if and only if the fire consumes the whole graph. For a detailed account of this algorithm, we refer to~\cite[Section 5.1]{BS13} and~\cite{Dhar90}.

\subsection{Jacobians of wedge sums of graphs}
Given two graphs with distinguished vertices $(G_1,v_1)$ and $(G_2,v_2)$, the \textit{wedge sum} is the graph formed by identifying $v_1$ and $v_2$. We suppress the dependency on the choice of distinguished vertices in what follows, as the choice will not matter, denoting the wedge sum as $G_1\vee G_2$. A key tool in our proof is the fact that the Jacobian of a wedge sum of graphs is the orthogonal direct sum of the Jacobians.

\begin{proposition}
\label{prop:wedge_sum}
Let $G_1$, $G_2$ be graphs. Then
$$ \Jac(G_1 \vee G_2) \cong \Jac(G_1) \oplus \Jac(G_2), $$
where $\oplus$ denotes the orthogonal direct sum of finite abelian groups with pairing.
\end{proposition}

\begin{proof}
This follows from the fact that any piecewise linear function on $G$ corresponds to a piecewise linear function on $G_i$ by restriction, and conversely any function on $G_i$ can be extended to a function on $G$ by giving it a constant value on $G \smallsetminus G_i$.
\end{proof}

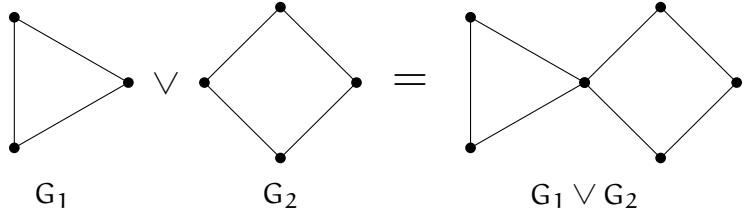
\begin{figure}[h!]
  \begin{tikzpicture}
    \draw {(0:1) -- (120:1) -- (240:1)} -- cycle (270:1.2) node[below]
    {$G_1$};
    \foreach \theta in {0, 120, 240} {
      \fill (\theta:1) circle (2pt) ;
    } ;

    \node at (0:1.5) {\large $\vee$};

    \begin{scope}[xshift=3cm]
      \draw {(0:1) -- (90:1) -- (180:1) -- (270:1)} -- cycle (270:1.2)
      node[below] {$G_2$};

      \foreach \theta in {0, 90, ...,  270} {
        \fill (\theta:1) circle (2pt) ;
      } ;
    \end{scope}

    \node at (0:4.7) {\LARGE $=$} ;

    \begin{scope}[xshift=6cm]
      \draw {(0:1) -- (120:1) -- (240:1)} -- cycle ;
      \foreach \theta in {0, 120, 240} {
        \fill (\theta:1) circle (2pt) ;
      } ;
      \begin{scope}[xshift=2cm]
        \draw {(0:1) -- (90:1) -- (180:1) -- (270:1)} -- cycle ;

        \foreach \theta in {0, 90, ...,  270} {
          \fill (\theta:1) circle (2pt) ;
        } ;
      \end{scope}
      \node[below] at (1,-1.2) {$G_1 \vee G_2$} ;
    \end{scope}
  \end{tikzpicture}
  \caption{The wedge sum operation on graphs. In this case, $\Jac(G_1)
  \cong \ZZ/3\ZZ$, $\Jac(G_2) \cong \ZZ/4\ZZ$, and $\Jac(G_1 \vee G_2)
  \cong \ZZ/12\ZZ$.}
\end{figure}

\subsection{Structure results for groups with pairing}
Our arguments will rely heavily on the classification of finite abelian groups with pairing from \cite{Miranda1984,Wall63}.  A first step in this classification is the following.

\begin{lemma}
\label{lem:relatively-prime}
Let $\Gamma$ be a group with pairing $\langle \cdot , \cdot \rangle$, and suppose that there exist subgroups $\Gamma_1 , \Gamma_2 \subseteq \Gamma$ such that $\Gamma \cong \Gamma_1 \times \Gamma_2$ as groups.  If the orders of $\Gamma_1$ and $\Gamma_2$ are relatively prime, then $\Gamma$ is isomorphic to the orthogonal direct sum $\Gamma_1 \oplus \Gamma_2$.
\end{lemma}

Lemma \ref{lem:relatively-prime} reduces the classification of finite abelian groups with pairing to the classification of $p$-groups with pairing.  In light of Proposition \ref{prop:wedge_sum}, this lemma allows us to focus on constructing graphs whose Jacobian is a given $p$-group with pairing.

If $p$ is an odd prime, then there are precisely two isomorphism classes of pairings on $\ZZ/p^r\ZZ$, for $r\geq 1$.  More precisely, every nondegenerate pairing on $\ZZ/p^r\ZZ$ is of the form
$$ \langle x,y \rangle_a = \frac{axy}{p^r} $$
for some integer $a$ not divisible by $p$. Two such pairings $\langle \cdot , \cdot \rangle_a, \langle \cdot , \cdot \rangle_b$ are isomorphic if and only if the Legendre symbols of $a$ and $b$ are equal.  We will refer to these two pairings as the \textit{residue} and \textit{nonresidue} pairings.  The following is a fundamental result for groups with pairing.

\begin{theorem}
If $p$ is an odd prime, then every finite abelian $p$-group with pairing decomposes as an orthogonal direct sum of cyclic groups with pairing.
\end{theorem}

When $p=2$, the situation is somewhat more intricate.  Up to isomorphism, there are 4 distinct isomorphism classes of pairings on $\ZZ/2^r\ZZ$, which we refer to as the \textit{non-exceptional pairings}. These are given below.

  \[ \AP \cong (\ZZ/2^r\ZZ, \langle \cdot , \cdot \rangle ), r \ge 1; \quad
  \langle x,y \rangle = \frac{xy}{2^r}\]
  \[ \BP \cong (\ZZ/2^r\ZZ, \langle \cdot , \cdot \rangle ), r \ge 2; \quad
  \langle x,y \rangle = \frac{-xy}{2^r}\]
  \[ \CP \cong (\ZZ/2^r\ZZ, \langle \cdot , \cdot \rangle ), r \ge 3; \quad
  \langle x,y \rangle = \frac{5xy}{2^r}\]
  \[ \DP \cong (\ZZ/2^r\ZZ, \langle \cdot , \cdot \rangle ), r \ge 3; \quad
  \langle x,y \rangle = \frac{-5xy}{2^r} . \]

In addition, on $(\ZZ/2^r\ZZ)^2$ there are two isomorphism classes of pairings that do not decompose as an orthogonal direct sum of cyclic groups with pairing.  We refer to these as the \textit{exceptional pairings}:

  \[ \EP \cong ((\ZZ/2^r\ZZ)^2, \langle \cdot , \cdot \rangle ), r \ge 1; \quad
  \langle e_i , e_j \rangle =
  \begin{cases}
    0, & i = j\\
    \frac{1}{2^{r}}, & \textrm{otherwise }
  \end{cases}
   \]
   \[ \FP \cong ((\ZZ/2^r\ZZ)^2, \langle \cdot , \cdot \rangle ), r \ge 2; \quad
   \langle e_i , e_j \rangle =
   \begin{cases}
     \frac{1}{2^{r-1}}, & i = j\\
     \frac{1}{2^{r}}, & \textrm{otherwise } ,
   \end{cases}
   \]
where $e_i$ and $e_j$ are generators for $(\ZZ/2^r\ZZ)^2$.

We note the following two results of Miranda~\cite{Miranda1984}.

\begin{lemma}
\label{lem:2group_pairing_mod8}
Let $\Gamma$ be a finite abelian group of order $2^r$, with pairing $\langle \cdot , \cdot \rangle$. If $\langle x,x \rangle = \frac{a}{2^r}$ for some $x \in \Gamma$ and odd positive integer $a$, then $\Gamma$ is cyclic generated by $x$.  Furthermore, for some $c \in \{\pm 1, \pm 5\}$, with $c \equiv a \pmod 8$, there is an isomorphism of groups $\phi:\Gamma \to \ZZ/2^r\ZZ$ such that
  \begin{equation*}
    \langle x,y \rangle = \frac{c\phi(x)\phi(y)}{2^r}.
  \end{equation*}
\end{lemma}




\begin{theorem}
The groups $\AP , \BP , \CP , \DP , \EP , \FP$ generate all 2-groups with pairing under orthogonal direct sum.
\end{theorem}

\section{Odd groups with pairing}\label{sec: odd-groups}

In this section, we investigate which groups with pairing of odd order occur as the Jacobian of a graph.  The decomposition of the Jacobain of a wedge sum as the orthogonal sum of the Jacobians of its components reduces our goal to the following.

\noindent
{\bf Problem.} Given a pairing $\langle\cdot,\cdot\rangle$ on the group $\ZZ/p^r\ZZ$ with $p$ odd, find a graph $G$ such that $Jac(G)$ is isomorphic to $\ZZ/p^r\ZZ$, such that $\langle\cdot,\cdot\rangle$ is induced by the monodromy pairing.

When $p=2$, which we consider in Section~\ref{sec: 2-groups}, we must also consider the non-decomposable pairings on $\ZZ/2^r\ZZ \times \ZZ/2^r\ZZ$.

\subsection{Subdivided Banana Graphs} \label{sec:subdividedbanana}
We begin with the following construction.

\begin{construction}
\textnormal{Let $\bm s = (s_1,\ldots, s_m)$ be a tuple of positive integers. Let $B_m$ denote the so-called ``banana graph'', which has two vertices and $m$ edges between them. Construct the \textit{$\bm s$-subdivided banana graph} from $B_m$ by subdividing the $i$th edge $s_i-1$ times. We denote this graph by $B_{\bm s}$, see Figure~\ref{fig: subdivided-banana}.}
\end{construction}

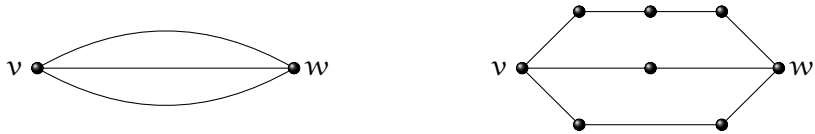
\begin{figure}[h!]
\begin{tikzpicture}[scale=1.5]
\coordinate (1) at (0,0);
\coordinate (2) at (0.5,0.5);
\coordinate (3) at (1.125,0.5);
\coordinate (4) at (1.75,0.5);
\coordinate (5) at (2.25,0);
\coordinate (6) at (1.125,0);
\coordinate (7) at (0.5,-0.5);
\coordinate (8) at (1.75,-0.5);
\node at (-0.2,0) {$v$};
\node at (2.45,0) {$w$};

\coordinate (A) at (-2,0);
\coordinate (B) at (-4.25,0);
\node at (-1.8,0) {$w$};
\node at (-4.45,0) {$v$};

\draw (1)--(2)--(3)--(4)--(5)--(6)--(1)--(7)--(8)--(5);

\path (A) edge [bend left] (B);
\path (A) edge [bend right] (B);
\path (A) edge (B);

\draw[ball color = black] (1) circle (0.5mm);
\draw[ball color = black] (2) circle (0.5mm);
\draw[ball color = black] (3) circle (0.5mm);
\draw[ball color = black] (4) circle (0.5mm);
\draw[ball color = black] (5) circle (0.5mm);
\draw[ball color = black] (6) circle (0.5mm);
\draw[ball color = black] (7) circle (0.5mm);
\draw[ball color = black] (8) circle (0.5mm);
\draw[ball color = black] (A) circle (0.5mm);
\draw[ball color = black] (B) circle (0.5mm);
\end{tikzpicture}
\caption{The $3$-banana graph and the subdivided banana $B_{(4,2,3)}$.}
\label{fig: subdivided-banana}
\end{figure}

\begin{proposition}
\label{prop:banana_pairing}
Fix a prime $p$ and an integer $r$.  Let $\bm s = (s_1,\ldots, s_m)$ be a tuple of positive integers such that
$$ \sum_{i=1}^m \frac{\prod_{j=1}^m s_j}{s_i} = p^r $$
and $\gcd (s_i, p) = 1$ for all $i$.  Then
  \begin{equation*}
    \Jac(B_{\bm s}) \cong (\ZZ/p^r\ZZ, \langle \cdot,\cdot \rangle),
  \end{equation*}
  where $\langle \cdot,\cdot \rangle$ is the pairing on $\ZZ/p^r\ZZ$ given by
  \begin{equation*}
    \langle x,y \rangle = \frac{\left(\prod_{i=1}^m s_i \right)xy}{p^r}.
  \end{equation*}
\end{proposition}

\begin{proof}
We first show that $\vert \Jac(B_{\bm s}) \vert = p^r$.  Every spanning tree of $B_{\bm s}$ is obtained by deleting one edge each from all but one of the subdivided edges of $B_m$.  It follows that the number of spanning tees of $B_{\bm s}$ is
$$ \sum_{i=1}^m \frac{\prod_{j=1}^m s_j}{s_i} = p^r . $$

We now show that $\Jac(B_{\bm s})$ is cyclic by exhibiting a generator.  Let $v$ and $w$ be the two vertices of $B_{\bm s}$ of valence $m$ pictured in Figure~\ref{fig: subdivided-banana}, and consider the divisor $D = v-w$.  Note that the order of $D$ must be a power of $p$, and let $t \leq r$ be the smallest nonnegative integer such that $p^t D$ is equivalent to $0$.  By definition, there exists a function $f: V(G) \to \ZZ$ such that $\ddiv (f) = p^t D$.

Orient the graph so that the head of each edge points toward $w$, and for each edge $e$ with head $x$ and tail $y$, let $b(e) = f(x) - f(y)$.  Since $D(v)= 0$ for any $v \in V(G) \smallsetminus \{ v,w \}$, we must have $b(e_1) = b(e_2)$ for any two edges in the same subdivided edge of $B_m$, and we may therefore write $b_i = b(e)$ for any edge $e$ in the $i$th subdivided edge.  Observe that $b_i s_i = f(w)-f(v)$ for all $i$.  As $\ddiv (f) = p^t D$, we may conclude that $\sum_{i=0}^m b_i = p^t$.  Consequently,
  \begin{equation*}
    p^t = \sum_{i=1}^m \frac{f(w)-f(v)}{s_i} =
    \frac{(f(w)-f(v)) p^r}{\prod_{i=1}^m s_i} .
  \end{equation*}
  From this, we deduce
  \begin{equation*}
    \prod_{i=1}^m s_i = p^{r-t} (f(w)-f(v)) .
  \end{equation*}
Since $\gcd (s_i, p) = 1$ for all $i$, this is impossible unless $r = t$, and thus the group is cyclic, generated by $D$.

The monodromy pairing on $\Jac(B_{\bm s})$ is fully determined by the value of $\langle D,D \rangle$.  Consider a function $f: V(G) \to \ZZ$ such that $b_i = \frac{\prod_{j=1}^m s_j}{s_i}$.  We see that $\ddiv (f) = p^r D$, and hence $\langle D,D \rangle = \frac{\prod_{i=1}^m s_i}{p^r}$.
\end{proof}

\begin{remark}
We have recently become aware that Proposition~\ref{prop:banana_pairing} was proven earlier in \cite[Section 2]{Lorenzini2000}.  We nevertheless reprove it here, as the argument is simple and the banana graph $B_{\bm s}$ is central to our later constructions.
\end{remark}

The cycle graph $C_n$ and the banana graph $B_n$ are both special cases of the subdivided banana. The following is an immediate corollary.

\begin{corollary}
\label{cor: PairingBananaAndCycle}
For any prime $p$ and integer $r$,
  \begin{equation*}
    \Jac(B_{p^r}) \cong (\ZZ/p^r\ZZ, \langle \cdot , \cdot \rangle_1)
  \end{equation*}
  \begin{equation*}
    \Jac(C_{p^r}) \cong (\ZZ/p^r\ZZ, \langle \cdot , \cdot \rangle_{-1}) ,
  \end{equation*}
  where $\langle \cdot , \cdot \rangle_1$ and $\langle \cdot , \cdot \rangle_{-1}$ are the
  pairings on $\ZZ/p^r\ZZ$ given by
  \begin{equation*}
    \langle x,y \rangle_1 = \frac{xy}{p^r} \qquad \langle x,y \rangle_{-1} = \frac{(-1)xy}{p^r} .
  \end{equation*}
\end{corollary}


\subsection{Results on quadratic residues}

Observe that the monodoromy pairing on $\Jac (B_{p^r})$ is the residue pairing on $\ZZ/p^r\ZZ$.  To achieve the nonresidue pairing, we will use the subdivided banana graph $B_{\bm s}$ for an appropriate choice of $\bm s$.  Our approach will rely on quadratic reciprocity, and it will be necessary to consider the cases $p \equiv 1 \pmod 4$ and $p \equiv 3 \pmod 4$ separately.  

\begin{proposition}
  \label{prop:q_bound}
  For any sufficiently large prime $p$, there exists a prime quadratic nonresidue $q \equiv 3 \pmod 4$, such that $q$ is less than $2\sqrt{p}$.
\end{proposition}

\begin{proof}
Let $\chi_1$ be the nontrivial character mod $4$ and $\chi_2$ the quadratic character mod $p$, and let $\mathbb{X}$ be the group of Dirichlet characters generated by $\chi_1$ and $\chi_2$.  The group $\mathbb{X}$ has conductor $f=\mathrm{lcm}(4,p) = 4p$ and exponent dividing $n=2$.  Define the form
\[
\chi = 1 + \chi_1 \chi_2 - \chi_1 - \chi_2 .
\]
By \cite[Theorem 1.4]{Pollack}, there exists an odd prime
\[
q_2 \ll (4p)^{\frac{1}{4}+\epsilon} f^{\epsilon} \ll 2p^{\frac{1}{4} + 2\epsilon}
\]
such that $\chi (q_2) \neq 0$.  By construction, however, if $\chi (q_2) \neq 0$ then $\chi_1 (q_2) = \chi_2 (q_2) = -1$.  It follows that $q_2$ is a quadratic nonresidue and $q_2 \equiv 3 \pmod 4$.
\end{proof}

We will also need the following proposition

\begin{proposition}  \label{prop:large-r-bound}
For any sufficiently large prime p and integer $r > 1$, there exist nonresidues $q_1 = 1 \mod 4$, $q_2 = 3 \mod 4$ with $q_1, q_2 < 2\sqrt{p^r}$.
\end{proposition}

\begin{proof}
As in the previous proof, let $\chi_1$ be the nontrivial character mod $4$ and $\chi_2$ the quadratic character mod $p$. To ask for a prime quadratic nonresidue $q \equiv 3 \mod 4$ is to ask for a prime $q$ such that $\chi_1(q) = \chi_2(q) = -1$. Consider the abelian field extension $K$ of $\QQ$ given by $K = \QQ(\sqrt{-1},\sqrt{\alpha})$, where
\[
\alpha = (-1)^{\frac{p-1}{2}}p.
\]
The extension $K$ is degree $4$ with conductor $4p$. The characters $\chi_1$ and $\chi_2$ are quadratic, and thus we may apply~\cite[Theorem 1.7]{Pollack}, to obtain an upper bound on the prime $q$,
\[
q\ll 2p^{\frac{1}{2}+\epsilon}.
\]
Now for the $1 \mod 4$ case, we simply replace $\chi_1(q)  = \chi_2(q) = -1$ above with the conditions
\[
\chi_1(q) = 1, \ \chi_2(q) = -1.
\]
and apply~\cite[Theorem 1.7]{Pollack} again.
\end{proof}

\begin{proposition}[Conditional on GRH]
  \label{prop:q_bound_grh}
For any prime $p>10^9$, there exists a prime quadratic nonresidue $q \equiv 3 \pmod 4$ such that $q < 2\sqrt{p}$.
\end{proposition}
\begin{proof}
Let $\alpha = (-1)^{\frac{p-1}{2}}p$, and let $K = \mathbb{Q} (\sqrt{-1}, \sqrt{\alpha})$.  The degree of the extension $K/\mathbb{Q}$ is 4, and the discriminant is $(4p)^2$.  By \cite[Theorem 5.1]{BachSorenson}, by assuming GRH, that there exists a prime quadratic nonresidue $q \equiv 3 \pmod 4$ satisfying
\[
q < (8 \log(4p) + 15)^2 .
\]
The term on the right is smaller than $2\sqrt{p}$ as long as $p>10^9$.
\end{proof}

Given a prime $q$ that satisfies the bounds above, we will need to find a particular way to write it as a sum of two positive integers, to ensure that $\bm s$ has the desired properties.  Below, we check that such a decomposition exists, and that this decomposition provides the properties we require.

\begin{lemma}
\label{lem:decompose_q}
Let $q$ be an odd prime, and let $k$ be an integer such that $\left(\frac{k}{q}\right) = \left(\frac{-1}{q}\right)$. Then there
  exists $0 < a < q$ such that $a(q-a) \equiv k \pmod q$.
\end{lemma}

\begin{proof}
Consider the set
  \begin{equation*}
    R_q = \left\{\ell \in \mathbb{F}_q : \left(\frac{\ell}{q}\right) =
    \left(\frac{-1}{q}\right)\right\},
  \end{equation*}
and the map $\phi: \mathbb{F}_q \to \mathbb{F}_q$ given by $\phi(x) = -x^2$.  The image of $\phi$ must be a subset of $R_q$.  For a fixed $a$, the polynomial $x^2 + a$ has at most two roots in $\mathbb{F}_q$.  Since $\vert R_q \vert = \frac{q - 1}{2}$, $\phi$ must therefore surject onto $R_q$.  Hence, there exists an integer $a$ such that $\phi(a) = k$, and we have $k \equiv -a^2 \equiv a(q-a) \pmod q$, as required.
\end{proof}

\begin{lemma}
\label{lem:exist_q}
Let $p$ be a sufficiently large prime with $p \equiv 1 \pmod 4$, and let $r$ be an integer.  Then there exists a prime $q$, with $\left(\frac{q}{p^r} \right) = -1$, and a positive integer $a < q$ such that the quantity
  \[\frac{p^r - a(q-a)}{q}\]
is a positive integer.
\end{lemma}

\begin{proof}
By Proposition \ref{prop:large-r-bound}, there exists a nonresidue $q$ with $\left(\frac{-1}{q}\right) = \left(\frac{p^r}{q}\right)$, and $\frac{q^2}{4} < p^r$.  By Lemma \ref{lem:decompose_q}, there exists a positive integer $a<q$ such that $p^r \equiv a(q-a) \pmod q$. Therefore $p^r - a(q-a)$ is positive and divisible by $q$.
\end{proof}

We now apply Lemma \ref{lem:exist_q} to establish the existence of an $\bm s$ such that $\Jac (B_{\bm S}) \cong \ZZ/p^r\ZZ$ with the nonresidue pairing.

\begin{proposition}
\label{prop:exist_decomp}
For any sufficiently large prime $p$ and integer $r$, there exists $\bm s = \{s_1, \ldots s_m\}$ such that
$$ \sum_{i=1}^m \frac{\prod_{j=1}^m s_j}{s_i} = p^r , $$
$\gcd(p, s_i) = 1$ for all $i$, and $\prod_{i=1}^m s_i$ is a nonresidue modulo $p$.
\end{proposition}

\begin{proof}
First consider the case that $p \equiv 3 \pmod 4$. Choose $\bm s = \{1,p^r-1\}$, and note that $p^r-1 \equiv -1 \pmod {p^r}$ is a nonresidue modulo $p^r$.

In the case that $p \equiv 1 \pmod 4$, let $q,a$ be as in Lemma \ref{lem:exist_q}, and let

  \[s_1 = a, \quad s_2 =  q-a, \quad s_3 = \frac{p^r - a(q-a)}{q} . \]

Since both $a$ and $q-a$ are smaller than $p$, they are relatively prime to $p$, and therefore the product $a(q-a)$ is relatively prime to $p$ as well. Now, the quantity $s_1s_2s_3$ is a nonresidue mod $p^r$ iff $\dfrac{(-1)(a(q-a))^2}{q}$ is a nonresidue mod $p$.  Since $p \equiv 1 \pmod 4$, $-1$ is a residue modulo $p^r$, and hence the numerator of this expression is also a residue.  Therefore $\left( \frac{s_1s_2s_3}{p^r} \right) = \left( \frac{q}{p^r} \right) = -1$, and the result follows.
\end{proof}

\subsection{Proof of Theorems~\ref{thm: nonconditional-pairing} and \ref{thm: conditional-pairing}}

\begin{proof} [Proof of Theorem \ref{thm: nonconditional-pairing}]
By Corollary \ref{cor: PairingBananaAndCycle}, $\Jac (B_{p^r}) \cong \ZZ/p^r\ZZ$ with the residue pairing.  By Propositions \ref{prop:banana_pairing} and \ref{prop:exist_decomp}, for any sufficiently large prime $p$ and integer $r \geq 1$, there exists an $\bm s$ such that $\Jac (B_{\bm s}) \cong \ZZ/p^r\ZZ$ with the nonresidue pairing.  By taking wedge sums of these graphs, we obtain all groups with pairing of odd order.
\end{proof}

Our proof of Theorem \ref{thm: conditional-pairing} is aided by the fact that in certain cases, we can explicitly construct an $\bm s$ satisfying the conditions required to achieve the nonresidue pairing:

\begin{proposition}
\label{prop:p_1mod24}
Let $p$ be an odd prime, not equivalent to $1 \pmod {24}$, and $r \geq 1$ an integer.  Then there exists an $\bm s$ such that
$$ \sum_{i=1}^m \frac{\prod_{j=1}^m s_j}{s_i} = p^r , $$
and $\prod_{i=1}^m s_i$ is a nonresidue modulo $p$.
\end{proposition}

\begin{proof}
  We consider the following three cases.
  \begin{enumerate}[(A)]
  \item When $p \equiv 3 \pmod 4$, as before, we may use $\bm s = \{1,p^r-1\}$.
  \item When $p \equiv 5 \pmod 8$, use $\bm s = \{1, 1, \frac{p^r-1}{2}\}$.  Since $p \equiv 1 \pmod 4$, the product $s_1s_2s_3$ is a nonresidue modulo $p$ iff $2$ is a nonresidue modulo $p$---which is the case when $p \equiv 5 \pmod 8$.
  \item When $p \equiv 2 \pmod 3$, if $p \equiv 3 \pmod 4$, we are in the first case above. Otherwise, we have $p \equiv 1 \pmod 4$, and $2$ is a nonresidue modulo $p$. Choose $\bm s = \{1, 1, \frac{p^r-1}{2}\}$ as before.
  \end{enumerate}

The only remaining possibility after eliminating these three cases is $p \equiv 1 \pmod {24}$.
\end{proof}

\begin{remark}
\label{rmk: computer}
Proposition \ref{prop:p_1mod24} shows that we could provide an unconditional proof of Theorem \ref{thm: conditional-pairing} if we could show that Proposition \ref{prop:q_bound_grh} holds for all primes $p \equiv 1 \pmod {24}$. In fact, computer search has verified that the proposition holds for all such primes smaller than $10^9$. The code is available upon request of the authors.
\end{remark}

\begin{proof} [Proof of Theorem \ref{thm: conditional-pairing}]
By Corollary \ref{cor: PairingBananaAndCycle}, $\Jac (B_{p^r}) \cong \ZZ/p^r\ZZ$ with the residue pairing.  By Propositions~\ref{prop:banana_pairing} and~\ref{prop:p_1mod24}, for any odd prime $p$ not congruent to $1 \pmod {24}$ and integer $r \geq 1$, there exists an $\bm s$ such that $\Jac (B_{\bm s}) \cong \ZZ/p^r\ZZ$ with the nonresidue pairing.  By Propositions \ref{prop:q_bound_grh} and \ref{prop:exist_decomp}, if we assume GRH, then for any prime $p > 10^9$ and integer $r \geq 1$, there exists an $\bm s$ such that $\Jac (B_{\bm s}) \cong \ZZ/p^r\ZZ$ with the nonresidue pairing.  Finally, the  computer search referenced in Remark \ref{rmk: computer} shows that, for all primes $p \equiv 1 \pmod {24}$, $p<10^9$, there exists an $\bm s$ such that $\Jac (B_{\bm s}) \cong \ZZ/p^r\ZZ$ with the nonresidue pairing.  Using the wedge sum construction, we may obtain all groups with pairing of odd order, as desired.
\end{proof}


\section{2-groups with pairing}\label{sec: 2-groups}

We now turn to the task of constructing graphs $G$ for which $\Jac(G) \cong ((\ZZ/2^r\ZZ)^k, \langle \cdot , \cdot \rangle)$ for given positive integers $r$ and $k$, and pairing $\langle \cdot , \cdot \rangle$.  For each of the non-exceptional pairings on $\ZZ/2^r\ZZ$, we find a graph whose Jacobian is isomorphic to $\ZZ/2^r\ZZ$ with the given pairing.

\subsection{Multicycle graphs}

In addition to the subdivided banana graphs of Section~\ref{sec:subdividedbanana}, we will require one more construction.

\begin{construction}
\textnormal{Let $\bm s = (s_1, \ldots, s_m)$ be a tuple of positive integers.  Construct the \textit{$\bm s$-multicycle graph} $C_{\bm s}$ on the vertices $v_1 , \ldots , v_m$ by introducing $s_i$ edges between $v_i$ and $v_{i+1}$ (here $i$ is taken mod $m$), see Figure~\ref{fig: multicycle}.}
\end{construction}

\begin{figure}[h!]
\begin{tikzpicture}
\coordinate (1) at (0,0);
\coordinate (2) at (2,0);
\coordinate (3) at (2,-2);
\coordinate (4) at (0,-2);

\draw[ball color=black] (1) circle (0.5mm);
\draw[ball color=black] (2) circle (0.5mm);
\draw[ball color=black] (3) circle (0.5mm);
\draw[ball color=black] (4) circle (0.5mm);

\draw (1) edge[me=2] (4);
\draw (1) edge[me=1] (2);
\draw (2) edge[me=3] (3);
\draw (3) edge[me=4] (4);

\node at (-0.1,0.2) {\tiny $v_1$};
\node at (2.1,0.2) {\tiny $v_2$};
\node at (2.1,-2.2) {\tiny $v_3$};
\node at (-0.1,-2.2) {\tiny $v_4$};
\end{tikzpicture}
\caption{The $C_{(1,3,4,2)}$ multicycle graph.}
\label{fig: multicycle}
\end{figure}
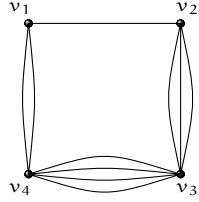

Note that the graphs $B_{\bm s}$ and $C_{\bm s}$ are planar duals of each other, and thus by Theorem \ref{thm:graph_dual_iso}, $\Jac (B_{\bm s}) \cong \Jac (C_{\bm s})$ as groups, but not necessarily as groups with pairing.

We now show that all of the cyclic 2-groups with non-exceptional pairing are realizable as Jacobians of graphs.

\begin{theorem}
Let $\Gamma \cong (\ZZ/2^r\ZZ, \langle \cdot , \cdot \rangle)$.  Then there exists a graph $G$ such that $\Jac(G) \cong \Gamma$.
\end{theorem}

\begin{proof}
Observe that, by Corollary \ref{cor: PairingBananaAndCycle}, $\Jac(B_{2^r}) \cong \AP$ and $\Jac(C_{2^r}) \cong \BP$. It remains to find constructions for graphs providing the groups $\CP$ and $\DP$.

By Lemma \ref{lem:2group_pairing_mod8}, it suffices to find graphs $G_1$ and $G_2$, with $\Jac(G_1) \cong \Jac(G_2) \cong \ZZ/2^r\ZZ$, such that for some $D_1 \in \Jac(G_1)$ and $D_2 \in \Jac(G_2)$, we have
  \begin{align*}
    \langle D_1 , D_1 \rangle_1 &= \frac{a}{2^r} \\
    \langle D_2 , D_2 \rangle_2 &= \frac{b}{2^r} ,
  \end{align*}
where $a \equiv 3 \pmod 8$ and $b \equiv -3 \pmod 8$.

We consider the cases for even and odd $r$ separately.  For odd $r$, let $\bm s = \{1, 2, \frac{2^r - 2}{3}\}$, and let $G_1 = B_{\bm s}$, $G_2 = C_{\bm s}$.

  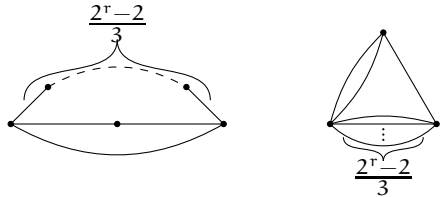
\begin{figure}[h!]
    \begin{center}
      \begin{tikzpicture}[scale=0.7]
        \draw (0,0) to (4,0) ;
        \draw (0,0) to [out=-30, in=-150] (4,0) ;
        \draw (0,0) to (.7,.7) ;
        \draw (3.3,.7) to (4,0) ;
        \draw [dashed] (.7,.7) to [out=30, in=150] (3.3,.7) ;

        \draw [fill] (0,0) circle [radius=.5mm] ;
        \draw [fill] (4,0) circle [radius=.5mm] ;
        \draw [fill] (2,0) circle [radius=.5mm] ;
        \draw [fill] (.7,.7) circle [radius=.5mm] ;
        \draw [fill] (3.3,.7) circle [radius=.5mm] ;

        \draw (.25,.45) to  [out=90, in=-90] (2,1.6) ;
        \draw (3.75,.45) to  [out=90, in=-90] (2,1.6) ;

        \node at (2,1.9) {$\frac{2^{r}-2}{3}$} ;

        \begin{scope}[xshift=6cm]
          \draw (0:0) [out=45, in=-105] to (60:2) ;
          \draw (0:0) [out=75, in=-135] to (60:2) ;
          \draw (0:0) [out=15, in=165] to (0:2) ;
          \draw (0:0) [out=0, in=180] to (0:2) ;
          \draw (0:0) [out=-45, in=-135] to (0:2) ;
          \draw (60:2) to (0:2) ;

          \draw [fill] (1, -.3) circle [radius=.01] ;
          \draw [fill] (1, -.2) circle [radius=.01] ;
          \draw [fill] (1, -.1) circle [radius=.01] ;

          \draw (.25,-.25) [out=-90, in=90] to (1,-.7) ;
          \draw (1.75,-.25) [out=-90, in=90] to (1,-.7) ;

          \draw [fill] (0, 0) circle [radius=.5mm] ;
          \draw [fill] (0:2) circle [radius=.5mm] ;
          \draw [fill] (60: 2) circle [radius=.5mm] ;

          \node at (1, -1) {$\frac{2^{r}-2}{3}$} ;
        \end{scope}
      \end{tikzpicture}
    \end{center}
    \caption{The graphs $B_{\bm s}$ and $C_{\bm s}$, for $\bm s = \{1,2,\frac{2^r - 2}{3}\}$}
  \end{figure}

Consider a function $f: V(B_{\bm s}) \to \ZZ$, given by
  \begin{align*}
    v_0 &\mapsto 0 \\
    v_0' &\mapsto 2 \\
    v_{21} &\mapsto 1 \\
    v_{3j} &\mapsto 2^n - 4 - j.
  \end{align*}

If $D_1 = v_{31} - v_0$, then $\ddiv(f) = 2^rD_1$.  It follows that $\langle D_1 , D_1 \rangle_1 = \frac{f(v_{31})}{2^r} = \frac{2^r - 3}{2^r}$, as required.

Now consider the function $f:V(C_{\bm s}) \to \ZZ$ given by
\[
v_0 \mapsto 0,\ \ v_1 \mapsto 2,\ \ v_2 \mapsto 3.
\]
If $D_2 = v_2 - v_0$, then $\ddiv(f) = 2^rD_2$, so $\langle D_2 , D_2 \rangle_2 = \frac{3}{2^r}$, as desired.

For even $r$, let $\bm s = \{1, 1, 1, \frac{2^r - 1}{3}\}$, and again let $G_1 = B_{\bm s}$ and $G_2 = C_{\bm s}$.

  \begin{figure}[h!]
    \begin{center}
      \begin{tikzpicture}[scale=0.7]

        \draw (0,0) to (4,0) ;
        \draw (0,0) to [out=-30, in=-150] (4,0) ;
        \draw (0,0) to [out=30, in=150] (4,0) ;
        \draw (0,0) to (.7,.7) ;
        \draw (3.3,.7) to (4,0) ;
        \draw [dashed] (.7,.7) to [out=30, in=150] (3.3,.7) ;

        \draw [fill] (0,0) circle [radius=.5mm] ;
        \draw [fill] (4,0) circle [radius=.5mm] ;

        \draw [fill] (.7,.7) circle [radius=.5mm] ;
        \draw [fill] (3.3,.7) circle [radius=.5mm] ;

        \draw (.25,.45) to  [out=90, in=-90] (2,1.6) ;
        \draw (3.75,.45) to  [out=90, in=-90] (2,1.6) ;

        \node at (2,2) {$\frac{2^{r}-1}{3}$} ;

        \begin{scope}[xshift=6cm]
          \draw (0,0) to (0,2) ;
          \draw (0,0) to (2,0) ;
          \draw (2,0) to (2,2) ;
          \draw (0,2) to (2,2) ;
          \draw (0,2) to [out=-15, in=-165] (2,2) ;
          \draw (0,2) to [out=45, in=135] (2,2) ;

          \draw [fill] (1, 2.1) circle [radius=.01] ;
          \draw [fill] (1, 2.2) circle [radius=.01] ;
          \draw [fill] (1, 2.3) circle [radius=.01] ;

          \draw (.2,2.25) [out=90, in=-90] to (1,2.7) ;
          \draw (1.8,2.25) [out=90, in=-90] to (1,2.7) ;

          \node at (1, 3.1) {$\frac{2^{r}-1}{3}$} ;

          \draw [fill] (0,0) circle [radius=.5mm] ;
          \draw [fill] (2,2) circle [radius=.5mm] ;
          \draw [fill] (2,0) circle [radius=.5mm] ;
          \draw [fill] (0,2) circle [radius=.5mm] ;
        \end{scope}
      \end{tikzpicture}
    \end{center}
    \caption{The graphs $B_{\bm s}$ and $C_{\bm s}$, for $\bm s = \{1, 1, 1, \frac{2^r
        - 1}{3}\}$}
  \end{figure}
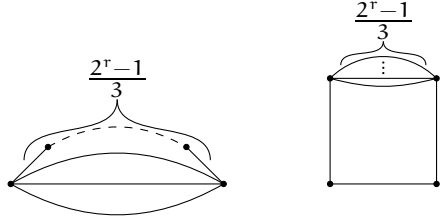

For the banana graph, we see from Proposition \ref{prop:banana_pairing} that $\Jac(B_{\bm s})$ is cyclic of order $2^r$, with pairing
  \begin{equation*}
    \langle x,y \rangle = \dfrac{\frac{2^r-1}{3}xy}{2^r}.
  \end{equation*}

For the multicycle graph, consider a function $f: V(C_S) \to \ZZ$, defined by $f(v_i) = i$.  If $D_2 = v_3 - v_0$, then $\ddiv(f) = -2^rD_2$, hence $\langle D_2 ,D_2 \rangle = \frac{3}{2^r}$, and the result follows.
\end{proof}

\subsection{$2$-groups with exceptional pairings}\label{sec: exceptional-pairings}

Each of the above constructions gives a graph with cyclic Jacobian, giving four of the six generators for 2-groups with pairing.  We have few concrete results concerning the exceptional pairings.  However, we make the following observation.

\begin{proposition}
For any $k \ge 1$, there is no graph $G$ such that $\Jac(G) \cong (\mathcal{E}_2)^k$.
\end{proposition}

\begin{proof}
This is a result of the characterization of graphs $G$ with $\Jac(G) \cong (\ZZ/2\ZZ)^{2k}$, given below in Remark \ref{remark:2group_characterize}.  Since the Jacobian of a cycle always gives rise to the group $\mathcal{A}_2$, any such graph has Jacobian $({A}_2)^{2k}$.
\end{proof}

This result, combined with our failure to find \textit{any} graph $G$ that yields the group $\EP$, leads us to make the following conjecture:

\begin{conjecture}
For any $k \ge 1$, there is no graph $G$ such that $\Jac(G) \cong (\EP)^k$.
\end{conjecture}

We note, however, that there do exist examples of graphs $G$ such that a \textit{subgroup} $H \subset \Jac(G)$ (with the restricted pairing) is isomorphic to $\EP$.  For example, $\Jac(B_{2,2,2}) \cong (\ZZ2\ZZ)^2 \times \ZZ/3\ZZ$, and by inspection we can see that the $2$-part with the restricted monodromy pairing is isomorphic to $\mathcal{E}_2$.

\begin{figure}[h!]
  \begin{center}
  \begin{tikzpicture}[scale=0.7]
    \draw (0,0) to [out=30, in=150]
    coordinate[pos=0.5] (A) (6,0);
    \draw (0,0) to [out=-30, in=-150]
    coordinate[pos=0.5] (B) (6,0);
    \draw (0,0) to (6,0);

    \draw [fill] (A) circle (0.5mm);
    \draw [fill] (B) circle (0.5mm);
    \draw [fill] (0,0) circle (0.5mm);
    \draw [fill] (6,0) circle (0.5mm);
    \draw [fill] (3,0) circle (0.5mm);

  \end{tikzpicture}
  \end{center}
  \caption{The graph $B_{2,2,2}$.}
\end{figure}
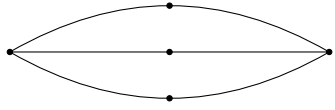

We have even fewer results regarding $\FP$. We note that the complete graph $K_4$ is a graph with Jacobian isomorphic to $\mathcal{F}_4$, but we were unable to find other examples of graphs that provide this pairing.

\section{Jacobians of simple graphs}
\label{sec:simple}

In this section, we consider which groups without a specified pairing occur as Jacobians of simple graphs.  If a finite abelian group $\Gamma$ does not have 2 as an invariant factor, then it is straightforward to construct a simple graph $G$ such that $\Jac (G) \cong \Gamma$, so this question is only interesting for groups of the form $(\ZZ/2\ZZ)^k \times H$.

\subsection{Preliminaries for proof of Theorem \ref{thm: z2k}} We first observe that any simple graph that has $2$ spanning trees must have a third.  To see this, consider the union of a spanning tree with a single edge not contained in the spanning tree.  This union contains a cycle, and the complement of any edge in this cycle is a spanning tree.  Since the graph is simple, however, this cycle must contain at least three edges.

Since the number of spanning trees is equal to the size of the Jacobian, there is no simple graph $G$ with $\Jac(G) \cong \ZZ/2\ZZ$.

Many of our arguments focus on the case where the graph $G$ is biconnected.  Recall that a graph $G$ is \textit{biconnected} if for any vertex $v\in V(G)$, the induced subgraph on $V(G)\setminus \{v\}$ is connected. In particular, if $G$ is not biconnected, then by definition, there is a vertex $v$ such that the induced subgraph on $V(G) \setminus \{ v \}$ is not connected.  The graph $G$ is therefore the wedge sum of the connected components, which implies that $\Jac(G)$ splits as a direct product of Jacobians.

\begin{definition}
Given a graph $G$, we write $\mu(G)$ for the maximum order of an element of $\Jac(G)$, and $\delta(G)$ for the maximum valency of a vertex in $G$.  When the graph $G$ is clear from context, we will simply write $\delta$ and $\mu$.
\end{definition}

\begin{lemma}
\label{lem:delta_le_mu}
For any biconnected graph $G$, $\delta(G) \le \mu(G)$.  Furthermore, if $\delta(G) = \mu(G)$, then $G$ must be the banana graph $B_{\mu}$.
\end{lemma}

\begin{proof}
The statement is immediate if $G$ consist of a single vertex, so we assume that $G$ has at least $2$ vertices. Let $v$ be a vertex in $V(G)$ with valency $\delta$, and let $w$ be a vertex adjacent to $v$. Consider the divisor $D = v-w$, and let $m < \delta$ be a positive integer. We apply Dhar's burning algorithm to check that $mD$ is $w$-reduced.  From the biconnectivity of $G$, we deduce that there is a path from $w$ to each of the neighbors of $v$ that does not contain $v$. Thus, each of the neighbors of $v$ is burned.  By definition, $\val (v) > m$, so it is burned as well. This means that $mD$ cannot be equivalent to $0$ as $0$ is the unique reduced divisor equivalent to $0$. It follows that $D$ has order at least $\delta$.

In the case that $\delta = \mu$, we must have $\delta D \sim 0$.  Starting from $\delta D$, chip-fire $v$ once to obtain a divisor $E$.  Applying the burning algorithm and the biconnectivity condition once more, we see that $v$, as well as each of its neighbors, must be burned, so that $E$ is $w$-reduced.  $E$ must therefore be the zero divisor, which is only possible if the multiplicity of the edge $\{v,w\}$ is $\delta$, i.e. $G$ is a banana graph.
\end{proof}

Recall that the \textit{genus} of a graph $G$ is its first Betti number, given by $g = |E(G)|-|V(G)|+1$.

\begin{corollary}
\label{cor:genus_v_mu}
For any biconnected graph $G$ with genus $g$ and $|V(G)| = n$,
  \begin{equation*}
    n \ge \frac{2g - 2}{\mu - 2}.
  \end{equation*}
\end{corollary}

\begin{proof}
Let $e$ be the total number of edges in $G$. We have an inequality
  \begin{equation*}
    2e = \sum_{i=1}^n \val(v_i) \le \sum_{i=1}^n \delta = n \cdot \delta
    \le n \cdot \mu.
  \end{equation*}
Since $e = g + n - 1$, we see that $2g - 2 \le n \cdot (\mu - 2)$.
\end{proof}

We are now ready to prove Theorem~\ref{thm: z2k}.

\begin{proof}[Proof of Theorem \ref{thm: z2k}]
Let $G$ be a simple graph with $\Jac(G) \cong (\ZZ/2\ZZ)^k$.  We may assume that $G$ has no vertices of valence 1, because the graph obtained by contracting the edge adjacent to such a vertex has isomorphic Jacobian.  If $G$ is not biconnected, then $G$ decomposes as a wedge sum, and $\Jac (G)$ decomposes as a direct sum of Jacobians, one of which must be isomorphic to $(\ZZ/2\ZZ)^r$ for some positive integer $r \leq k$.  We may therefore assume that $G$ is biconnected.  By Lemma \ref{lem:delta_le_mu}, it also has no vertices of valence 3 or greater.  It follows that $G$ is a cycle.  Since $\Jac (C_n ) \cong \ZZ/n\ZZ$, we must have $n=2$, which means $G$ cannot be simple.
\end{proof}

\begin{remark}
\label{remark:2group_characterize}
The proof of Theorem \ref{thm: z2k} also gives a complete characterization of graphs $G$ with $\Jac(G) \cong (\ZZ/2\ZZ)^k$. In general, we can always obtain such a graph by the following procedure. Start with a tree $T$, and choose a subset of $k$ edges of $T$. Construct a new graph $G$ from $T$ by doubling each edge in this subset. See Figure~\ref{fig:ProductsOfZ2}.
\end{remark}

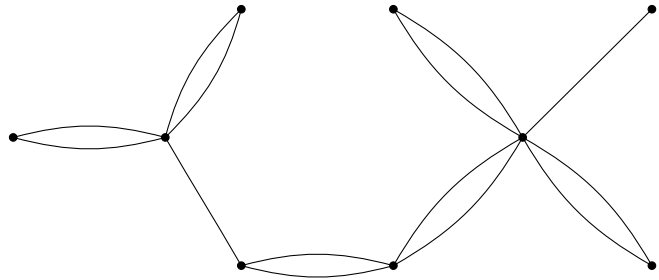
\begin{figure}[h!]
  \begin{center}
    \begin{tikzpicture}

      \draw (0,0) to [out=15, in=165] (2,0) ;
      \draw (0,0) to [out=-15, in=-165] (2,0) ;

      \draw (2,0) to [out=75, in=225] (3, 1.7) ;
      \draw (2,0) to [out=45, in=255] (3, 1.7) ;

      \draw (2,0) to [out=-60, in=-240] (3, -1.7) ;

      \draw (3,-1.7) to [out=15, in=165] (5, -1.7) ;
      \draw (3,-1.7) to [out=-15, in=-165] (5, -1.7) ;

      \draw (5, -1.7) [out=60, in=210] to (6.7, 0) ;
      \draw (5, -1.7) [out=30, in=240] to (6.7, 0) ;

      \draw (6.7, 0) [out=120, in=-30] to (5, 1.7) ;
      \draw (6.7, 0) [out=150, in=-60] to (5, 1.7) ;

      \draw (6.7, 0) to (8.4, 1.7) ;

      \draw (6.7, 0) to [out=-30, in=120] (8.4, -1.7) ;
      \draw (6.7, 0) to [out=-60, in=150] (8.4, -1.7) ;

      \draw [fill] (0,0) circle [radius=0.5mm] ;
      \draw [fill] (2,0) circle [radius=0.5mm] ;
      \draw [fill] (3,1.7) circle [radius=0.5mm] ;
      \draw [fill] (3,-1.7) circle [radius=0.5mm] ;
      \draw [fill] (5,-1.7) circle [radius=0.5mm] ;
      \draw [fill] (6.7, 0) circle [radius=0.5mm] ;
      \draw [fill] (5,1.7) circle [radius=0.5mm] ;
      \draw [fill] (8.4,1.7) circle [radius=0.5mm] ;
      \draw [fill] (8.4,-1.7) circle [radius=0.5mm] ;
    \end{tikzpicture}
    \caption{An example of a graph $G$ with $\Jac(G) \cong
      (\ZZ/2\ZZ)^6$}
      \label{fig:ProductsOfZ2}
  \end{center}
\end{figure}

\subsection{Preliminaries: Proof of Theorem \ref{thm:too-many-twos}}
Our next goal is to generalize Theorem \ref{thm: z2k} to graphs whose Jacobian is of the form $(\ZZ/2\ZZ)^k \times H$.  We begin with the following bound on the genus of $G$.

\begin{proposition}\cite[Proposition 5.2]{Lorenzini1989}
\label{prop:genus_cycle}
If $G$ is a graph of genus $g$ and $\Jac (G) \cong (\ZZ/2\ZZ)^k \times H$, then $g \ge k$.
\end{proposition}

Applying Corollary \ref{cor:genus_v_mu} to this result shows that
\begin{equation*}
\vert V(G) \vert \ge \frac{2k-2}{\mu - 2}
\end{equation*}

We require the following result about lengths of paths in $G$.

\begin{lemma}
\label{lem:2valent_path}
Let $G$ be a biconnected graph, and suppose that there exists a path $P$ with vertices $\{v_1, \ldots, v_\ell\}$ on $G$ such that $\val(v_i) = 2$ for all $1 < i < \ell$.  Then $\Jac(G)$ contains an element of order at least $\ell$.
\end{lemma}

\begin{proof}
Let $m < \ell$, and consider $D = v_2 - v_1$.  As $G$ is biconnected, there is a path from $v_1$ to $v_{m+1}$ that does not contain any of the vertices of $P$. Dhar's burning algorithm shows that $v_{m+1} - v_1$ is the $v_1$-reduced divisor equivalent to $mD$, and hence $mD \nsim 0$ for $m < \ell$.
\end{proof}

Our approach will now be to establish an upper bound on $\vert V(G) \vert$ in terms of $\mu$ and $\vert H \vert$, and then use this to obtain an upper bound on $k$.

\begin{proposition}
\label{prop:v_bound}
For any finite abelian group $H$, there exists an integer $n_H$ such that, for any biconnected simple graph $G$ with $\Jac(G) \cong (\ZZ/2\ZZ)^k \times H$, we have $\vert V(G) \vert < n_H$.
\end{proposition}

\begin{proof}
Let $U = \{u \in V(G) : \val(u) > 2\}$.  We will first establish a bound on $m = \vert U \vert$, and then bound $\vert V(G) \vert$ in terms of $m$.

Fix a vertex $u \in U$, and consider the set of divisors $\mathcal{U} = \{u_i-u \vert u_i \in U \}$.  For any $D_1 \neq D_2 \in \mathcal{U}$, we claim that $2D_1 - 2D_2 = 2u_1 - 2u_2$ is $u_2$-reduced.  Since $G$ is biconnected, there is a path from $u_2$ to each of the neighbors of $u_1$ that does not contain $u_1$.  Applying Dhar's burning algorithm, we see that since $\val (u_2) > 2$, the entire graph will be burned.  Therefore $2D_1 - 2D_2$ is $u_2$-reduced, hence $2D_1 \nsim 2D_2$.

We now define a map
  \begin{align*}
    \varphi: \Jac (G) &\to \Jac (G) \\
    D &\mapsto 2D.
  \end{align*}

By the above, we have that the restriction of $\varphi$ to $\mathcal{U}$ is injective.  Furthermore, since $\vert \mathrm{im}(\varphi) \vert \le \vert H \vert$, we see that $m \le \vert H \vert$.

We now wish to bound $\vert V(G) \vert$ in terms of $m$.  To do so, we construct a new graph $G'$ from $G$, according to the following algorithm.

\begin{enumerate}[(1)]
\item Choose any vertex of $G$ of valency $2$. Delete it, and draw
an edge between its neighbors.
\item Repeat until there are no 2-valent vertices remaining.
\end{enumerate}

  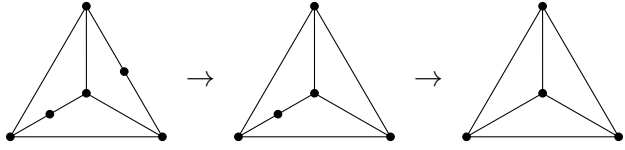
\begin{figure}[h!]
    \begin{center}
      \begin{tikzpicture}

        \draw (0:0) to (60:2) ;
        \draw (0:0) to (0:2) ;
        \draw (60:2) to (0:2) ;
        \draw (0:0) to (30:1.155) ;
        \draw (0:2) to (30:1.155) ;
        \draw (60:2) to (30:1.155) ;

        \draw [fill] (0:0) circle [radius=0.5mm] ;
        \draw [fill] (0:2) circle [radius=0.5mm] ;
        \draw [fill] (60:2) circle [radius=0.5mm] ;
        \draw [fill] (30:1.155) circle [radius=0.5mm] ;
        \draw [fill] (30:.6) circle [radius=0.5mm] ;
        \draw [fill] (30:1.73) circle [radius=0.5mm] ;

        \draw [shift={(3,0)}] (0:0) to (60:2) ;
        \draw [shift={(3,0)}] (0:0) to (0:2) ;
        \draw [shift={(3,0)}] (60:2) to (0:2) ;
        \draw [shift={(3,0)}] (0:0) to (30:1.155) ;
        \draw [shift={(3,0)}] (0:2) to (30:1.155) ;
        \draw [shift={(3,0)}] (60:2) to (30:1.155) ;

        \draw [shift={(3,0)}][fill] (0:0) circle [radius=0.5mm] ;
        \draw [shift={(3,0)}][fill] (0:2) circle [radius=0.5mm] ;
        \draw [shift={(3,0)}][fill] (60:2) circle [radius=0.5mm] ;
        \draw [shift={(3,0)}][fill] (30:1.155) circle [radius=0.5mm] ;
        \draw [shift={(3,0)}][fill] (30:.6) circle [radius=0.5mm] ;

        \draw [shift={(6,0)}] (0:0) to (60:2) ;
        \draw [shift={(6,0)}](0:0) to (0:2) ;
        \draw [shift={(6,0)}](60:2) to (0:2) ;
        \draw [shift={(6,0)}](0:0) to (30:1.155) ;
        \draw [shift={(6,0)}](0:2) to (30:1.155) ;
        \draw [shift={(6,0)}](60:2) to (30:1.155) ;

        \draw [shift={(6,0)}][fill] (0:0) circle [radius=0.5mm] ;
        \draw [shift={(6,0)}][fill] (0:2) circle [radius=0.5mm] ;
        \draw [shift={(6,0)}][fill] (60:2) circle [radius=0.5mm] ;
        \draw [shift={(6,0)}][fill] (30:1.155) circle [radius=0.5mm] ;

        \node at (2.5, .75) {$\rightarrow$} ;
        \node at (5.5, .75) {$\rightarrow$} ;

      \end{tikzpicture}
      \caption{The transformation $G \mapsto G'$}
    \end{center}
  \end{figure}

Note that even if $G$ is simple, $G'$ need not be.  It is clear, however, that $G$ and $G'$ have the same number of vertices with valency greater than $2$, and that $\delta(G) = \delta(G')$.

By Lemma \ref{lem:delta_le_mu}, we must have that $e' = \vert E(G') \vert$ is at most $m \cdot \mu$ (since otherwise there would necessarily be a vertex of $G$ with valency greater than $\delta$).  Each 2-valent vertex of $G$ is uniquely associated with some edge of $G'$.  If there are more than $(e' \cdot \mu)$ divalent vertices in $G$, then at least $\mu$ of them are associated with a single edge of $G'$.  In this case, $G$ would contain a path $P$ of length greater than $\mu$, where each vertex of $P$ has valency $2$.  This contradicts Lemma \ref{lem:2valent_path}, so we have
  \[
  \vert V(G) \vert - m < m\mu^2 .
  \]
If we let $n_H = \vert H \vert (1 + \mu^2)$, then $\vert V(G) \vert < n_H$.
\end{proof}

Applying Corollary \ref{cor:genus_v_mu} and Proposition \ref{prop:genus_cycle}, we see that for sufficiently large $k$, we must have $\vert V(G) \vert > n_H$.  This in turn implies that for sufficiently large $k$, $(\ZZ/2\ZZ)^k \times H$ is not the Jacobian of any biconnected simple graph. We will use this fact to show that this result holds generally, for all simple graphs.

\begin{proof}[Proof of Theorem \ref{thm:too-many-twos}]
We proceed by induction on $\vert H \vert$. When $\vert H \vert = 1$ or $2$, Theorem \ref{thm: z2k} gives the bound $k_H = 1$.  For $\vert H \vert \ge 3$, there must exist (by Proposition \ref{prop:v_bound}) an integer $k'$ such that, if $k > k'$ and $\Jac(G) \cong (\ZZ/2\ZZ)^k \times H$, then $G$ is not biconnected.

By the inductive hypothesis, for any proper subgroup $H' \subset H$, there exists an integer $k(H')$ such that for all $k > k(H')$, no simple graph $G'$ has $\Jac(G') \cong (\ZZ/2\ZZ)^k \times H'$.  Now, since $H$ is finite, there are finitely many pairs of nontrivial proper subgroups $H_1, H_2 \subset H$ such that $H_1 \times H_2 \cong H$. Define

\begin{equation*}
  k'' = \max\{k(H_1) + k(H_2) : H_1, H_2 \textnormal{ nontrivial}, H_1
  \times H_2 \cong H\}.
\end{equation*}

Now let $k_H = \max(k', k'')$.  We wish to show that for all $k > k_H$, if $\Jac(G) \cong (\ZZ/2\ZZ)^k \times H$, then $G$ is not simple.  Let $G$ be a graph with this Jacobian, and let $k > k_H$.  Since $k > k'$, $G$ is not biconnected, so it must be the wedge sum of two graphs $G_1$ and $G_2$. There must then exist integers $k_1, k_2$ with $k_1 + k_2 = k$ and groups $H_1, H_2$ with $H_1 \times H_2 \cong H$ such that

\begin{align*}
  \Jac(G_1) &\cong (\ZZ/2\ZZ)^{k_1} \times H_1, \\
  \Jac(G_2) &\cong (\ZZ/2\ZZ)^{k_2} \times H_2 .
\end{align*}

Without loss of generality, we may assume that neither $G_1$ nor $G_2$ is a tree, so that $\Jac(G_1)$ and $\Jac(G_2)$ are both nontrivial.  If either $H_1$ or $H_2$ are trivial, then $G_1$ (resp. $G_2$) would have Jacobian isomorphic to $(\ZZ/2\ZZ)^k$ for $k > 0$, contradicting Theorem \ref{thm: z2k}.

Finally, since $k_1 + k_2 = k > k'' \ge k(H_1) + k(H_2)$, we must have that either $k_1 > k(H_1)$ or $k_2 > k(H_2)$.  It follows that either $G_1$ or $G_2$ is not simple, so $G$ is not simple.
\end{proof}

\subsection{Further queries}

Analysis of the proof of Theorem \ref{thm:too-many-twos} suggests that, if $H \cong \ZZ/p^r\ZZ$ for some prime $p$, then $k_H = O(\vert H \vert p^3)$.  In
practice, it seems that much better bounds should hold.  For instance, we were unable to find any simple graph $G$ where $\Jac(G) \cong (\ZZ/2\ZZ)^k \times H$ for any $k > \vert H \vert$.

In some cases, it is possible to directly verify that certain groups do not arise as the Jacobian of any simple graph.  Recall that a graph is 2-edge-connected if it remains connected after the deletion of any edge.  For a given $m$, while there are infinitely many isomorphism classes of simple graphs with fewer than $m$ spanning trees, at most finitely many of these classes represent 2-edge-connected graphs.  This results from the fact that, for any vertex $v_0$ on a 2-edge-connected graph, any divisor of the form $v - v_0$ is $v_0$-reduced, and hence there are at least as many spanning trees on the graph as there are vertices.

By contracting bridges, any graph $G$ may be uniquely associated to a 2-edge-connected graph with isomorphic Jacobian.  For a given
group $H$, therefore, it is possible to compute the Jacobian of all 2-edge-connected simple graphs with at most $\vert H \vert$ spanning trees, and
verify that $H$ does or does not occur.

Computer searches of this nature have led to the following:

\begin{proposition}
The following groups are not isomorphic to the Jacobian of any simple graph:
\begin{itemize}
\item $\ZZ/2\ZZ \times \ZZ/4\ZZ$,
\item $(\ZZ/2\ZZ)^2 \times \ZZ/4\ZZ$,
\item $\ZZ/2\ZZ \times (\ZZ/4\ZZ)^2$.
\end{itemize}
\end{proposition}
The key fact in the proof of the nonoccurence of groups with many factors of $\ZZ/2\ZZ$ seems to be the requirement that $G$ is biconnected, rather than that $G$ is simple. It has been shown that, asymptotically, the probability that the Jacobian of a random graph is cyclic is relatively high~\cite{CKLPW}.  We expect that the Jacobians of most graphs have a small number of invariant factors. Since random graphs are highly connected, we conjecture the following.

\begin{conjecture}
For any positive integer $n$, there exists $k_n$ such that if $k > k_n$, there is no biconnected graph $G$ with $\Jac(G) \cong (\ZZ/n\ZZ)^k$.
\end{conjecture}

The conjecture follows from our results for $n=3$. To see this, observe from Lemma \ref{lem:delta_le_mu} that the only biconnected graphs with Jacobian $(\ZZ/3\ZZ)^k$ are the $3$-cycle and the $3$-banana. In this case, we have $k_3 = 1$.

\bibliographystyle{siam}
\bibliography{Jacobians}

\end{document}